\documentclass[11pt]{amsart}
\usepackage{amssymb}
\usepackage{verbatim}
\usepackage{amsmath}
\usepackage{}
\usepackage{amssymb,latexsym}
\usepackage{enumerate}

\newcommand{\dbar}{\ensuremath{\overline\partial}}

\newcommand{\C}{\ensuremath{\mathbb{C}}}

\makeatletter
\newcommand{\sumprime}{\if@display\sideset{}{'}\sum%
            \else\sum'\fi}
\makeatother

\begin{document}

\numberwithin{equation}{section}

% define theorem environments
\newtheorem{theorem}{Theorem}[section]
\newtheorem{proposition}[theorem]{Proposition}
\newtheorem{conjecture}[theorem]{Conjecture}
\def\theconjecture{\unskip}
\newtheorem{corollary}[theorem]{Corollary}
\newtheorem{lemma}[theorem]{Lemma}
\newtheorem{observation}[theorem]{Observation}
\newtheorem{definition}{Definition}
\numberwithin{definition}{section} %\def\thedefinition{\unskip}
\newtheorem{remark}{Remark}
\def\theremark{\unskip}
\newtheorem{kl}{Key Lemma}
\def\thekl{\unskip}
\newtheorem{question}{Question}
\def\thequestion{\unskip}
\newtheorem{example}{Example}
\def\theexample{\unskip}
\newtheorem{problem}{Problem}

\thanks{Research supported by Knut and Alice Wallenberg Foundation, and the China Postdoctoral Science Foundation.}

\address{School of Mathematical Sciences, Fudan University, Shanghai, 200433, China}
\address{Current address: Department of Mathematical Sciences, Chalmers University of Technology and
University of Gothenburg. SE-412 96 Gothenburg, Sweden}
\email{wangxu1113@gmail.com}
\email{xuwa@chalmers.se}

\title[Curvature of the generalized Hodge metric]{Curvature restrictions on a manifold with a flat Higgs bundle}
 \author{Xu Wang}
\date{\today}

\begin{abstract} We shall prove a semi-negative curvature property for a manifold with a flat admissible Higgs bundle.
\bigskip

\noindent{{\sc Mathematics Subject Classification} (2010): 32A25, 53C55.}

\smallskip

\noindent{{\sc Keywords}: Higgs bundle, variation of Hodge structure, Calabi-Yau manifold, Hodge metric, curvature of subbundle.}
\end{abstract}
\maketitle

\section{Introduction}

Our motivation to write this paper is to generalize a result (see \cite{Lu99} or \cite{Lu01}, see also \cite{GS69} and \cite{Deligne69} for early results, \cite{PP14} and references therein for recent results) of Griffiths-Schmid-Deligne-Lu (on the curvature property of the Hodge metric associated to a variation of Hodge structure) to the Higgs bundle case. The notion of Higgs bundle is introduced by Hitchin in \cite{Hitchin87} for the one dimensional case and by Simpson in \cite{Simpson92} for the general case. There is a natural class of Higgs bundles coming from variation of Hodge structures, where the Higgs field is defined by using the Kodaira-Spencer map (see the next section). For a variation of Hodge structure, it is known that the holomorphic sectional curvature of the base manifold is bounded above by a negative constant (see \cite{GS69} and \cite{Deligne69}) if the associated period map is an immersion (thus the Hodge metric is well defined, see \cite{Lu99}, \cite{Lu01}, \cite{LS04} and \cite{FL05}). In \cite{Lu99}, Lu proved that if the period map is an immersion then the Hodge metric is K\"ahler (which may not be true in the mixed case, see \cite{PP14}). It is natural to ask whether the above results are true for general Higgs bundles. In this paper, we shall prove that at least Lu's curvature property of the Hodge metric can be generalized to the Higgs bundle case. More precisely, we shall define the Hodge semi-metric associated to a general Higgs bundle. And we call a Higgs bundle is admissible if the associated Hodge semi-metric is a Hermitian metric. Then our main result can be stated as follows:

\begin{theorem}\label{th:main}
 If there is a flat admissible Higgs bundle, say $H$, over a complex manifold, say $B$, then the Hodge metric on $B$ is a K\"ahler metric with semi-negative holomorphic bisectional curvature. Assume further that the $H$ is $k$-nilpotent (see Definition \ref{de:nilpotent-higgs}). Then the holomorphic sectional curvature of $B$ is bounded above by $-(k^2 {\rm Rank}(H))^{-1}$.
\end{theorem}

The known proofs for the variation of Hodge structure case depend on the curvature property of the locally homogeneous complex manifold (see \cite{GS69}) and  the curvature formula for the subbundle. In this paper, we shall show in the third section that the curvature formula for the subbundle will be enough to prove our main theorem (this idea has also been used by Simpson for other purpose, see page 27 in \cite{Simpson92}).

\section{Basic notions on the Higgs bundle}

\begin{definition}[Higgs bundle, by Simpson \cite{Simpson92}] \label{de:higgs} A Higgs bundle is a pair $(H,\theta)$, where $H$ is a holomorphic vector bundle over a complex manifold $B$ and $\theta$ is an ${\rm End}(H)$-valued  holomorphic one form such that $\theta^2\equiv0$ on $B$.  
\end{definition}

\textbf{Remark: associated bundle map}. By using the Higgs field $\theta$, one may define a holomorphic bundle map from the tangent bundle $T_B$ of $B$ to ${\rm End}(H)$ as follows
\begin{equation}
\eta: v\mapsto v\lrcorner ~\theta:= \theta_v, \  v\in T_B.
\end{equation}
Then we have
\begin{equation}\label{eq:theta20}
\theta^2\equiv 0 \ \text{iff}\ [\theta_v, \theta_w]\equiv 0, \ \forall \ v, w\in T_B.
\end{equation}
We shall also use the following definition:

\begin{definition}[Admissible Higgs bundle]\label{de:ad-higgs} A Higgs bundle $(H,\theta)$ is said to be \textbf{admissible} if the associated bundle map $\eta$ is an injection from $T_B$ to ${\rm End}(H)$.
\end{definition}

\textbf{Example: Higgs bundle associated to a family of holomorphic vector bundles}. It is known that there is a natural Higgs bundle structure on the base manifold of a proper holomorphic fibration. More precisely, let $\pi$ be a proper holomorphic submersion from a complex manifold $\mathcal X$ to a complex manifold $B$ with connected fibres $X_t:=\pi^{-1}(t)$. Let us denote by
\begin{equation}
\kappa: v\mapsto \kappa(v) \in H^{0,1}(T_{X_t}).
\end{equation}
the Kodaira-Spencer map. Let $E$ be a holomorphic vector bundle over the total space $\mathcal X$. If the following direct image sheaves 
\begin{equation}
\mathcal H_E^k :=\oplus_{p+q=k} \mathcal H^{p,q}_E, \ \mathcal H^{p,q}_E:=R^q \pi_*\mathcal O(E\otimes \wedge^p T^*_{\mathcal X/B}).
\end{equation}
are locally free then $\kappa$ defines a natural holomorphic bundle map, say $\eta$, from $T_B$ to ${\rm End}(\mathcal H^k_E)$ (see \cite{BPW16}) such that $\eta(v) (\mathcal H^{p,q}_E)\subset \mathcal H^{p-1,q+1}_E$. Moreover, we have (see \cite{BPW16} or the appendix):
\begin{equation}
[\eta(v), \eta(w)]\equiv 0,
\end{equation}
which implies that the ${\rm End}(\mathcal H^k_E)$-valued holomorphic one form, say $\theta$,  associated to $\eta$ satisfies that $\theta^2\equiv 0$. Thus $(\mathcal H^k_E, \theta)$ is a Higgs bundle over the base manifold $B$. In case $E$ is flat and the total space $\mathcal X$ is K\"ahler, we know that (see \cite{G84}) $\eta$ is just the \textbf{differential of the period map}. Thus in that case, $(\mathcal H^k_E, \theta)$ is admissible iff the period map is an immersion. 

\medskip

\textbf{Remark: associated Hermitian metric on the base manifold $B$}.  If $(H,\theta)$ is an admissible Higgs bundle then $T_B$ can be seen as a holomorphic subbundle of ${\rm End}(H)$. Let us fix a smooth Hermitian metric, say $h$, on $H$. If $(H,\theta)$ is admissible then one may define a natural Hermitian structure on the subbundle $T_B$ of ${\rm End}(H)$ as follows:
\begin{equation}\label{eq:Hodge-higgs}
(v, w)_H:= (\theta_v,\theta_w)_{h\otimes h^*}, \ \ \forall \ v,w\in T_B.
\end{equation}
here we identify ${\rm End}(H)$ with $H\otimes H^*$. If $B$ is the base manifold of a Calabi-Yau family then the above metric is just the Hodge metric introduced by Lu \cite{Lu99}. In general, we shall introduce the following definition:

\begin{definition}[Hodge metric]\label{de:Hodge-higgs} Let $(H,\theta)$ be an admissible Higgs bundle. Then we call the Hermitian metric  defined by \eqref{eq:Hodge-higgs} the \textbf{Hodge metric} on $B$.
\end{definition}

\textbf{Remark: relation with the Weil-Petersson metric}. In \cite{LS04}, Lu and Sun proved that the Hodge metric associated to a smooth family of projective Calabi-Yau manifolds (with fibre dimension less than $5$) can be written as a linear combination of the associated Weil-Petersson metric and its Ricci form. In \cite{FL05}, Fang and Lu further proved a very interesting formula connecting the Weil-Petersson metric, the BCOV torsion and the generalized Hodge metric associated to a general smooth family of projective Calabi-Yau manifolds.

\medskip

\textbf{Remark: Hodge semi-metric}. If $(H,\theta)$ is not admissible then  \eqref{eq:Hodge-higgs} defines a semi-metric on $B$, we call it the Hodge semi-metric.

\medskip

\textbf{Remark: Gauss-Manin connection}. If the total space $\mathcal X$ is K\"ahler and $(E,h)$ is flat then we know that there is a natural flat connection (so called Gauss-Manin connection, see \cite{G84}), say $D^H$, on $(\mathcal H^k_E, \theta)$. Moreover, it is known that (see \cite{BPW16} or the appendix)
\begin{equation}
D^H=D^h+\theta+\theta^*,
\end{equation}
where $D^h:= \dbar+\partial^h$ denotes the Chern connection (with respect to the natural $L^2$-metric on $\mathcal H^k_E$) on the Higgs bundle $\mathcal H^k_E$ and $\theta^*$ denotes the adjoint of $\theta$. If we write $\theta$ locally as
\begin{equation}\label{eq:theta_j}
\theta=\sum dt^j\otimes \theta_j, \ \theta_j:= \partial/\partial t^j \lrcorner~ \theta.
\end{equation}
Then each $\theta_j$ is a local holomorphic section of ${\rm End} (E)$. Let $\theta_j^*$ be the adjoint of $\theta_j$ with respect to $h$. Then we have $\theta^*:=\sum d\bar t^j\otimes \theta_j^*$.
In general, we shall define:

\begin{definition}[Flat Higgs bundle, by Simpson \cite{Simpson92}]\label{de:GGM} Let $(H,\theta)$ be a Higgs bundle with a smooth Hermitian metric $h$. We call $D^H:=D^h+\theta+\theta^*$ the \textbf{Higgs connection} on $H$. We call $(H,\theta,h)$ a \textbf{flat Higgs bundle} if the Higgs curvature $\Theta^H:=(D^H)^2\equiv 0$ on $B$.
\end{definition}

From the above definition, we know that if the total space $\mathcal X$ is K\"ahler and $(E,h)$ is flat then for each $k$, $(\mathcal H^k_E, \theta)$ is a flat Higgs bundle.

\medskip

In order to state our main theorem, we shall also define the nilpotent Higgs bundle. Let $(H,\theta)$ be a Higgs bundle. By iteration, one may also define $\eta^j$ as a bundle map from $\otimes^j T_B$ to ${\rm End}(H)$, i.e.
\begin{equation}
\eta^j(v_1\otimes \cdots\otimes v_j):=\theta_{v_1}\circ\cdots \circ \theta_{v_j}.
\end{equation}
If $H=\mathcal H^k_E$ then each $\eta^j$ is just the iterated Kodaira-Spencer action (see \cite{BPW16} or the appendix) and $\eta^{k+1}\equiv 0$. In general, we shall introduce the following notation:

\begin{definition}\label{de:nilpotent-higgs} Let $k$ be a natural number. A Higgs bundle $(H,\theta)$ is said to be $k$-nilpotent if the associated bundle map satisfies $\eta^{k+1}\equiv 0$. 
\end{definition}

\textbf{Remark: relation with a system of Hodge bundles}. If $(H,\theta)$ is a system of Hodge bundles (see page 44 in \cite{Simpson92}) then it is  $k$-nilpotent for some $k$. In particular, $(\mathcal H^k_E, \theta)$ is $k$-nilpotent. 

\section{Curvature of the Hodge metric}

We shall prove Theorem \ref{th:main} in this section. The K\"ahler part follows from the following proposition:

\begin{proposition} The fundamental form of the Hodge semi-metric associated to a Higgs bundle $(H,\theta)$ is $d$-closed if the Higgs curvature of $(H,\theta)$ has no $(2,0)$-part.
\end{proposition}

\begin{proof} By definition, we know that the fundamental form of the Hodge semi-metric associated to $(H,\theta)$ can be written as
\begin{equation}
\omega=i\{ \theta, \theta\},
\end{equation}
where $\{\cdot, \cdot\}$ denotes the sesquilinear product on ${\rm End}(H)$. Since $\theta$ is holomorphic, we have
\begin{equation}
\partial\omega=i \{\partial^{{\rm End}(H)} \theta, \theta\},
\end{equation}
where $\partial^{{\rm End}(H)} $ denotes the $(1,0)$-part of the Chern connection on ${\rm End}(H)$. Since ${\rm End}(H) \simeq H\otimes H^*$, by definition of the Chern connection on the dual bundle and the tensor product bundle, we know that
\begin{equation}
\partial^{{\rm End}(H)} \theta=[\partial^h, \theta].
\end{equation}
Notice that $[\partial^h, \theta]$ is just the $(2,0)$-part of the Higgs curvature. The proof is complete.
\end{proof}

Now let us prove the second part of Theorem \ref{th:main}. By Definition \ref{de:ad-higgs}, if $(H,\theta)$ is an admissible Higgs bundle then the associated bundle map
\begin{equation}
\eta: T_B\to {\rm End}(H), \ \ v\mapsto \theta_v,
\end{equation}
defines a holomorphic subbundle structure on $T_B$ and the Hodge metric is just the induced metric on $T_B$.  Let us denote by $D^{T_B}$ the Chern connection on $T_B$ with respect to the Hodge metric. Locally one may write 
\begin{equation}
D^{T_B}=\sum dt^j\otimes D^{T_B}_j +\sum d\bar t^j\otimes \dbar_{t^j}, \ \Theta_{j\bar k}^{T_B}:=[D^{T_B}_j, \dbar_{t^k}].
\end{equation}
By the curvature formula for the subbundle, we have
\begin{equation}\label{eq:curvature-123}
(\Theta_{j\bar k}^{T_B} v, w)_H=(\Theta_{j\bar k}^{{\rm End}(H)} \theta_v, \theta_w)- \left( P^{\bot} (D^{{\rm End}(H)}_j \theta_v), P^{\bot} (D^{{\rm End}(H)}_k \theta_w) \right),
\end{equation}
where $\Theta_{j\bar k}^{{\rm End}(H)}$ are the Chern-curvature operators on ${\rm End}(H)$ and $P^{\bot}$ denotes the orthogonal projection to the orthogonal complement of $T_B$ in ${\rm End}(H)$. We shall show that \eqref{eq:curvature-123} implies the following proposition:

\begin{proposition} Let $(H, \theta)$ be a flat admissible Higgs bundle over a complex manifold $B$. Then we have
\begin{equation}\label{eq:Griffiths}
\sum (\Theta_{j\bar k}^{T_B} v, v)_H\xi^j\bar\xi^k \leq 0, \  \forall \ v\in T_B, \ \xi\in \C^{\dim B},
\end{equation}
and 
\begin{equation}\label{eq:less-Nakano}
\sum(\Theta_{j\bar k}^{T_B} \partial/\partial t^k, \partial/\partial t^j)_H \leq 0.
\end{equation}
\end{proposition}

\begin{proof} Since ${\rm End}(H)\simeq H\otimes H^*$, by the curvature formula of the dual bundle and the tensor product bundle, we have
\begin{equation}
\Theta_{j\bar k}^{{\rm End}(H)} \theta_v=[\Theta^h_{j\bar k}, \theta_v],
\end{equation}
where $\Theta^h=(D^h)^2$ denotes the Chern connection on $(H,h)$. Thus by \eqref{eq:curvature-123}, we have
\begin{equation}\label{eq:curvature-1234}
(\Theta_{j\bar k}^{T_B} v, w)_H=([\Theta^h_{j\bar k}, \theta_v], \theta_w)- \left( P^{\bot} (D^{{\rm End}(H)}_j \theta_v), P^{\bot} (D^{{\rm End}(H)}_k \theta_w) \right).
\end{equation}
Since $(H, \theta)$ is flat, we have $\Theta^h_{j\bar k}=[\theta_k^*, \theta_j]$ (see \eqref{eq:theta_j} for the definition of $\theta_j$). Moreover, $\theta^2\equiv 0$ imples that $[\theta_v, \theta_j] \equiv 0$. Thus we have
\begin{equation}
[\Theta^h_{j\bar k}, \theta_v]=[[\theta_k^*, \theta_j], \theta_v]=[[\theta_k^*, \theta_v], \theta_j].
\end{equation}
Notice that
\begin{eqnarray*}
([[\theta_k^*, \theta_v], \theta_j], \theta_w) & = & \left([\theta_k^*, \theta_v]\theta_j-\theta_j[\theta_k^*, \theta_v], \theta_w\right) \\
& = & ([\theta_k^*, \theta_v], \theta_w \theta_j^*-\theta_j^*\theta_w) \\
& = & -([\theta_k^*, \theta_v], [\theta_j^*, \theta_w]).
\end{eqnarray*}
Thus we have
\begin{equation}
(\Theta_{j\bar k}^{T_B} v, w)_H= -([\theta_k^*, \theta_v], [\theta_j^*, \theta_w])- \left( P^{\bot} (D^{{\rm End}(H)}_j \theta_v), P^{\bot} (D^{{\rm End}(H)}_k \theta_w) \right),
\end{equation}
which implies \eqref{eq:Griffiths} and 
\begin{equation}
\sum(\Theta_{j\bar k}^{T_B} \partial/\partial t^k, \partial/\partial t^j)_H= -||\sum [\theta_k^*, \theta_k]||^2- \left( P^{\bot} (D^{{\rm End}(H)}_j \theta_k), P^{\bot} (D^{{\rm End}(H)}_k \theta_j) \right).
\end{equation}
But notice that $D^{{\rm End}(H)}_j \theta_k=[\partial^h_j, \theta_k]$ and flat-ness of our Higgs bundle implies that $[\partial^h_j, \theta_k]=[\partial^h_k, \theta_j]$. Thus
\begin{equation}\label{eq:thetajkjk}
\sum(\Theta_{j\bar k}^{T_B} \partial/\partial t^k, \partial/\partial t^j)_H= -||\sum [\theta_k^*, \theta_k]||^2-\sum ||P^{\bot} [\partial^h_j, \theta_k]||^2,
\end{equation}
which implies \eqref{eq:less-Nakano}. The proof is complete.
\end{proof}

\textbf{Remark}: \eqref{eq:Griffiths} implies the second part of Theorem \ref{th:main}.

\medskip

Now let us prove the following proposition, which implies the last part of Theorem \ref{th:main}.

\begin{proposition} Let $(H, \theta)$ be a flat admissible $k$-nilpotent Higgs bundle over a complex manifold $B$. Then for each $j$,
\begin{equation}\label{eq:sectional-curvature}
(\Theta_{j\bar j}^{T_B} \partial/\partial t^j, \partial/\partial t^j)_H \leq -(k^2{\rm Rank} (H))^{-1} ||\partial/\partial t^j||_H^4.
\end{equation}
\end{proposition}

\begin{proof} By \eqref{eq:thetajkjk}, we have
\begin{equation}
(\Theta_{j\bar j}^{T_B} \partial/\partial t^j, \partial/\partial t^j)_H \leq  -||[\theta_j^*, \theta_j]||^2
\end{equation}
Fix $t\in B$. Then $A:=\theta_j(t)$ is a $\mathbb C$-linear transform from the fibre $V:=H_t$ to itself. Since $(H, \theta)$ is $k$-nilpotent, we know that $A^{k+1}=0$. Using the Jordan normal form of $A$, we know that $V$ can be written as a direct sum of $(k+1) \ \mathbb C$-linear subspaces,
\begin{equation}
V=\oplus_{0\leq p\leq k} V_p, 
\end{equation}
such that
\begin{equation}
A(V_p)\subset V_{p+1}, \ A(V_k)=\{0\}.
\end{equation}
For each $0\leq p\leq k-1$, let us denote the following $\mathbb C$-linear map
\begin{equation}
A: V_p \to V_{p+1},
\end{equation}
by $A_p$. Put $A_{-1}=A_k=0$. By a direct computation, we have
\begin{equation}
|| [\theta^*_j, \theta_j]|| \geq \frac{\sum_{0\leq p\leq k} |{\rm Trace}(A_p^* A_p)-{\rm Trace}(A_{p-1}^* A_{p-1})|}{\sqrt{{\rm Rank}(H)}}.
\end{equation}
Put $a_p={\rm Trace}(A_p^* A_p)$. Then each $a_p$ is non-negative, and
\begin{equation}
\sum  |a_p-a_{p-1}| =a_0+\sum_{1\leq p\leq k-1} |a_p-a_{p-1}| +a_{k-1}\geq \max\{a_0, \cdots, a_{k-1}\} \geq \frac1k \sum a_p.
\end{equation}
Moreover, by definition, we have
\begin{equation}
||\partial/\partial t^j||_H^2= \sum a_p.
\end{equation}
Thus \eqref{eq:sectional-curvature} is true.
\end{proof}

\section{Applications}

The following Griffiths-Schmid theorem (see \cite{GS69} and \cite{Lu99}) is a direct corollary of our main theorem.

\begin{theorem}
Let $\pi:\mathcal X\to B$ be a proper holomorphic submersion from a K\"ahler manifold $\mathcal X$ to a complex manifold $B$ with connected fibres $X_t$. Assume that the Higgs bundle 
\begin{equation*}
\mathcal H^k:=\oplus_{p+q=k} \{H^{p,q}(X_t)\}, \
\end{equation*}
is admissible (i.e., the associated period map is an immersion). Then the holomorphic sectional curvature of $B$ with respect to the Hodge metric is bounded above by a negative constant.
\end{theorem}

\textbf{Remark}:  In fact, our main theorem implies an effective version (see Theorem 4.3 in \cite{LS04} for the Calabi-Yau case) of Griffiths-Schmid theorem: the holomorphic sectional curvature of $B$ is no bigger than $-(k^2 b_k)^{-1}$, where $b_k:={\rm Rank}(\mathcal H^k)$ is the $k$-th Betti number of the fibres.

\medskip

\textbf{Remark: Calabi-Yau family}. In the above theorem, if we assume further that the canonical line bundle of each fibre is trivial then $\mathcal H^n$ ($n=\dim_{\mathbb C} X_t$) is admissible iff the Kodaira-Spencer map is injective. Thus the above theorem can be used to study the curvature property of the base manifold of a Calabi-Yau family (see \cite{Lu01}).

\section{Acknowledgement} 

I would like to thank Bo Berndtsson for many inspiring discussions and his comments on paper writing. I would also like to thank Zhiqin Lu for discussions  on the curvature properties of the base manifold of a Calaba-Yau family and his comments on this paper. Thanks are also given to Bo-Yong Chen and Qing-Chun Ji for their constant support and encouragement.

\section{Appendix: Higgs bundle associated to a family of holomorphic vector bundles}

In this section, we shall give a careful study of the example after Definition \ref{de:ad-higgs}. Let $\pi$ be a proper holomorphic submersion from a complex manifold $\mathcal X$ to a complex manifold $B$ with connected fibres $X_t:=\pi^{-1}(t)$. Let $E$ be a holomorphic vector bundle over the total space $\mathcal X$. Let $E_t$ be the restriction to the fibre, $X_t$, of $E$. Let $H^{p,q}(E_t)$ be the Dolbeault cohomology groups on $(X_t, E_t)$. Let us fix a positive integer $k$. We shall use the following assumption:

\medskip

\textbf{A}: The dimension of $H^{p,q}(E_t)$ does not depend on $t$ in $B$ for all $p,q$ such that $p+q=k$.

\medskip

By the base change theorem, we know that $\mathbf{A}$ implies that  for each $p,q$ with $p+q=k$, there is a holomorphic vector bundle structure on
\begin{equation}
\mathcal H^{p,q}_E:=\{H^{p,q}(E_t)\}_{t\in B},
\end{equation}
such that the sheaf of germs of holomorphic sections of $\mathcal H^{p,q}_E$ is just the $q$-th direct image sheaf $R^q \pi_*\mathcal O(E\otimes \wedge^p T^*_{\mathcal X/B})$. 
By Theorem 3.1 in \cite{Wang16}, there is also a direct way to construct the holomorphic vector bundle structure of $\mathcal H^{p,q}_E$ as follows:

\medskip

\textbf{Holomorphic vector bundle structure of $\mathcal H^{p,q}_E$}: Let $h$ be a smooth Hermitian metric on $E$. Let $\omega$ be a smooth Hermitian metric on the total space $\mathcal X$. Let $\dbar^t$ be the $\dbar$-operator on $X_t$. Then by the Hodge theory, every cohomology class in $H^{p,q}(E_t)$ has a unique representative in
\begin{equation}
\mathcal H^{p,q}_{E_t}:=\ker\dbar^t \cap \ker(\dbar^t)^*,
\end{equation}
where each $(\dbar^t)^*$ denotes the adjoint of $\dbar^t$ with respect to $\omega^t:=\omega|_{X_t}$ and $h^t:=h|_{E_t}$. Thus we can identify $H^{p,q}(E_t)$ with $\mathcal H^{p,q}_{E_t}$, and we shall write
\begin{equation}
\mathcal H^{p,q}_E:=\{\mathcal H^{p,q}_{E_t}\}_{t\in B}.
\end{equation}
By a theorem of Kodaira-Spencer (see Page 349 in \cite{KSp}), we know that $\mathbf A$ implies that $\mathcal H_E^{p,q}$ has a smooth complex vector bundle structure. More precisely, if $\mathbf A$ is true then for every $t_0\in B$, $u^{t_0}\in \mathcal H_{E_{t_0}}^{p,q}$, there is a smooth $E$-valued $(p,q)$-form, say $\mathbf u$, on $\mathcal X$ such that
\begin{equation}
i_{t_0}^* \mathbf u=u^{t_0}, \ i_t^*\mathbf u \in \mathcal H_{E_t}^{p,q}, \ \forall \ t\in B, 
\end{equation} 
where each $i_t$ denotes the inclusion map from $X_t$ to $\mathcal X$. Then the smooth vector bundle structure $\mathcal H_E^{p,q}$ can be defined as follows:

\begin{definition}\label{de:smooth} We call $u: t\mapsto u^t\in\mathcal H^{p,q}_{E_t}$ a smooth section of $\mathcal H_E^{p,q}$ if there exists a smooth $E$-valued $(p,q)$-form, say $\mathbf u$, on $\mathcal X$ such that $i_t^*\mathbf u=u^{t}$, $\forall \ t\in B$. And we call $\mathbf u$ a representative of $u$. 
\end{definition}

Let $\{t^j\}_{1\leq j\leq m}$ be a holomorphic local coordinate system on $B$. For a smooth section, say $u$, of $\mathcal H^{p,q}_E$, let us define
\begin{equation}
\dbar_{t^j} u:  t \to \mathbb H^t\left(i_t^* [\dbar, \delta_{\overline{V_j}}]  \mathbf{u}\right), \ \ [\dbar, \delta_{\overline{V_j}}]:=\dbar\delta_{\overline{V_j}}+ \delta_{\overline{V_j}}\dbar,  \ \ \delta_{\overline{V_j}}:= \overline{V_j}\lrcorner,
\end{equation}
where $\textbf{u}$ is an arbitrary representative of $u$, $\mathbb H^t$ denotes the orthogonal projection to $\mathcal H^{p,q}_{E_t}$ and $V_j$ is an arbitrary smooth $(1,0)$-vector field on $\mathcal X$ such that $\pi_*V_j=\partial/\partial t^j$. One may check that (see Theorem 3.1 in \cite{Wang16}) each $\dbar_{t^j} u$ is well defined (i.e. does not depend on the choice of $V_j$ and $\mathbf u$). Then we know that
\begin{equation}
D^{0,1}u :=\sum d\bar t^j \otimes \dbar_{t^j} u,
\end{equation}
can be seen as a $(0,1)$-component of a connection on $\mathcal H_E^{p,q}$. It is proved in  \cite{Wang16} that $D^{0,1}$ is integrable, i.e. $(D^{0,1})^2=0$. Thus by the vector bundle version of the Newlander-Nirenberg theorem, we know that $D^{0,1}$ defines a holomorphic vector bundle structure on $\mathcal H_E^{p,q}$. 

\medskip

\textbf{Higgs connection on $\mathcal H_E^k$}: Now we know that $\mathbf A$ implies that
\begin{equation}
\mathcal H^k_E:=\oplus_{p+q=k} \mathcal H^{p,q}_E,
\end{equation}
is a holomorphic vector bundle over our base manifold $B$. We claim that the Kodaira-Spencer map
\begin{equation}
\kappa: v\mapsto \kappa(v):=[i_t^*(\dbar V)]\in H^{0,1}(T_{X_t}), \ \forall \ v\in T_t B,
\end{equation}
(here $[i_t^*(\dbar V)]$ denotes the $\dbar$-cohomology class of the $\dbar$-closed $T_{X_t}$-valued $(1,0)$-form $i_t^*(\dbar V)$ on $X_t$ and $V$ is an arbitrary smooth $(1,0)$-vector field on $\mathcal X$ such that $\pi_*V=v$) can be used to define a holomorphic bundle map from $T_B$ to  ${\rm End} (\mathcal H^{p,q}_E)$. In fact, each Kodaira-Spencer class $\kappa(v)$ defines a natural Kodaira-Spencer action, say $\eta(v)$, such that $\eta(v)$ sends a cohomology class, say $[u]$, in $H^{p,q}(E_t)$ to a cohomology class $[i_t^*(\dbar V) u]$ in $H^{p-1,q+1}(E_t)$. On the harmonic space, we then have
\begin{equation}
\eta(v)u^t:=  \mathbb H^t\left(i_t^* [\dbar, \delta_{V}]  \mathbf{u}\right),
\end{equation}
where $u: t\mapsto u^t$ is a smooth section of $\mathcal H^{p,q}_E$, $\mathbf{u}$ is a representative of $u$ and $\mathbb H^t$ denotes the orthogonal projection to $\mathcal H^{p-1,q+1}_{E_t}$. Since the Kodaira-Spencer map is holomorphic, we know that 
\begin{equation}
\eta: v\mapsto \eta(v)\in {\rm End}(\mathcal H^k _E)
\end{equation}
is a holomorphic bundle map and  $[\eta(v), \eta(w)]=0$ (one may also prove this fact by computing the commutators of the Lie derivatives). Consider $\theta$ such that
\begin{equation}
v \lrcorner ~\theta:=\theta_v =\eta(v).
\end{equation}
Then $\theta$ is holomorphic and $\theta^2=0$. Thus $\theta$ defines a Higgs field on $\mathcal H^k _E$. Let us denote by $D^h$ the Chern connection on $\mathcal H^k _E$. Then the Higgs connection, say $D^H$, on $\mathcal H^k _E$ is just $D^h+\theta+\theta^*$. Let us write
\begin{equation}
D^h=\sum d\bar t^j\otimes \dbar_{t^j} +\sum dt^j \otimes \partial^h_{t^j}, \ \theta=\sum dt^j\otimes \theta_j, \ \theta^*:=\sum d\bar t^j\otimes \theta_j^*.
\end{equation}
The proof of Proposition 3.3 (a Lefschetz decomposition trick) in \cite{Wang16} implies the following proposition: 

\begin{proposition} If $\omega$ is a K\"ahler form on the total space $\mathcal X$ then for every smooth section $u$ (with representative $\mathbf u$) of $\mathcal H^k_E$, we have 
\begin{equation}
\partial^h_{t^j} u:  t \to (\partial^h_{t^j} u)(t)=\mathbb H^t(i_t^* [\partial^E, \delta_{V_j}]\mathbf{u}) , \ \ (\theta_j^* u)(t)= \mathbb H^t(i_t^* [\partial^E, \delta_{\overline{V_j}}]\mathbf{u}),
\end{equation}
where $\partial^E$ denotes the $(1,0)$-part of the Chern connection on $E$ and each $V_j$ denotes the unique $(1,0)$-vector field on $\mathcal X$ such that $\pi_* V_j =\partial/\partial t^j$ and $i_t^*(V_j \lrcorner ~ \omega)\equiv 0$ (i.e. each $V_j$ is the horizontal lift of $\partial/\partial t^j$ with respect to $\omega$).
\end{proposition}

Let us write our Higgs connection $D^H$ as 
\begin{equation}
D^H:=\sum d\bar t^j\otimes D_{\bar j} +\sum dt^j \otimes D_j.
\end{equation} 
By the above proposition, if $\mathcal X$ is K\"ahler then 
\begin{equation}
(D_{\bar j}u)(t) = \mathbb H^t(i_t^* [d^E, \delta_{\overline{V_j}}]\mathbf{u}), \  (D_{j}u)(t) = \mathbb H^t(i_t^* [d^E, \delta_{V_j}]\mathbf{u}),
\end{equation}
where $d^E:=\dbar+\partial^E$ is the Chern connection on $E$. Assume further that $\Theta^E=(d^E)^2=0$, i.e. $E$ is flat. Then each $\mathbb H^t$ is equal to the orthogonal projection to $\ker d^{E_t} \cap \ker (d^{E_t})^*$ (since by the Bochner-Kodaira-Nakano formula, in this case, each $\dbar^t$-harmonic space is equal to the $d^{E_t}$-harmonic space), which implies that
\begin{equation}
([D_j, D_{\bar k}]u)(t)=\mathbb H^t(i_t^* [[d^E, \delta_{V_j}], [d^E, \delta_{\overline{V_k}}]]\mathbf{u})=\mathbb H^t(i_t^* [[d^E, \delta_{V_j}], [d^E, \delta_{\overline{V_k}}]]\mathbf{u}).
\end{equation}
Since
\begin{equation}
[[d^E, \delta_{V_j}], [d^E, \delta_{\overline{V_k}}]]= [d^E, \delta_{[V_j, \overline{V_k}]}]+ \Theta^E(V_j, \overline{V_k}),
\end{equation}
and $\Theta^E=0$, we know that $[D_j, D_{\bar k}]u\equiv 0$. By a similar argument, we have  $[D_j, D_{k}]u\equiv 0$ and $[D_{\bar j}, D_{\bar k}]u\equiv 0$. Thus we get the following Griffiths theorem:

\begin{theorem} If $\mathcal X$ is K\"ahler and $E$ is flat then $(\mathcal H^k_E,\theta)$ is flat as a Higgs bundle.  Assume further that $E$ is trivial then the Higgs connection $D^{H}$ is just the Gauss-Manin connection.
\end{theorem}

Since $D^H=D^h+\theta+\theta^*$, we know that the above theorem implies the Griffiths formula $$(D^h)^2+ \theta\theta^*+\theta^*\theta=0,$$
for the curvature of $\mathcal H^k_E$ with respect to the Hermitan norm defined by the $L^2$-norm on each fibre.

\end{document}